\documentclass[12pt]{amsart}
\usepackage{fullpage,amssymb}
\newtheorem{theorem}{Theorem}[section]
\newtheorem{lemma}[theorem]{Lemma}
\newtheorem{corollary}[theorem]{Corollary}
\newtheorem{proposition}[theorem]{Proposition}
\newtheorem*{notation*}{Notation}
\newtheorem*{p*}{Proposition~\ref{h.s.o.p}}
\theoremstyle{definition}
\newtheorem{definition}[theorem]{Definition}
\newtheorem{example}[theorem]{Example}
\newtheorem{remark}[theorem]{Remark}
\newcommand{\M}{{\operatorname{Mat}}}

\newcommand{\C}{{\mathbb C}}

\newcommand{\Z}{{\mathbb Z}}

\newcommand{\K}{K}

\newcommand{\Hom}{\operatorname{Hom}}
\newcommand{\kar}{\operatorname{char}}
\newcommand{\spa}{\operatorname{span}}
\newcommand{\im}{\operatorname{im}}
\newcommand{\llangle}{\,\,(\!\!\!\!<\!}
\newcommand{\rrangle}{\!>\!\!\!\!)\,\,}
\newcommand{\Q}{{\mathbb Q}}

\newcommand{\crk}{\operatorname{crk}}
\newcommand{\ncrk}{\operatorname{ncrk}}
\newcommand{\rk}{\operatorname{rk}}
\newcommand{\trk}{\operatorname{trk}}
\newcommand{\brk}{\operatorname{brk}}
\usepackage{multirow}
\usepackage{amsmath}
\usepackage{comment}
\usepackage{tikz}
\usepackage{mathtools}
\usepackage{arydshln}
\usepackage{graphicx}
\newcommand*{\Scale}[2][4]{\scalebox{#1}{\ensuremath{#2}}}%

\title{On non-commutative rank and tensor rank}

\author{Harm Derksen and Visu Makam}

\thanks{The first author was supported by NSF grant DMS-1302032 and the second author was supported by NSF grant DMS-1361789}

\begin{document}

\begin{abstract}
We  study the relationship  between the commutative  and the non-commutative rank of a linear matrix. We give examples   that show that the ratio of the two ranks comes arbitrarily close to 2. Such examples can be used for giving lower bounds for the border rank of a given tensor. Landsberg used such techniques to give nontrivial equations for the tensors of border rank at most $2m-3$ in $K^m\otimes K^m\otimes K^m$ if $m$ is even. He also gave such equations for tensors of border rank at most $2m-5$ in $K^m\otimes K^m\otimes K^m$ if $m$ is odd.
Using  concavity of tensor blow-ups we show non-trivial equations for tensors of border rank  $2m-4$ in $\K^m \otimes \K^m \otimes \K^m$ for odd $m$ for any field $\K$ of characteristic 0.  We also give another proof of the regularity lemma by Ivanyos, Qiao and Subrahmanyam.
\end{abstract}

\maketitle

\section{Introduction}
\subsection{Linear matrices}
We fix an infinite field $\K$. A linear matrix (or matrix pencil) $A$ over $\K$ is a matrix whose coefficients are linear expressions in variables $t_1,t_2,\dots,t_m$. The commutative rank $\crk(A)$ is defined as the rank over the commutative function field $\K ( t_1,t_2,\dots,t_m )$. On the other hand, one can take $t_1,t_2,\dots,t_m$ to be independent non-commuting variables and compute the rank over the free skew field $\K\llangle t_1,t_2,\dots,t_m \rrangle$. This is called the non-commutative rank, and is denoted $\ncrk(A)$. See \cite{FR} for details. 
 
 Over a skew field, the rank of a matrix is defined as the (left) row rank, which is equal to the (right) column rank. In particular, adding left multiplied rows to other rows and right multiplied columns to other columns does not affect the rank. Note also that a square matrix is invertible over a skew field if and only if it is of full rank. We refer to \cite{GGOW,IQS} for several equivalent definitions of non-commutative rank.
 
The following well-known example shows that the commutative and non-commutative rank of a linear matrix may differ.
\begin{example} \label{sk3}
This example is based on skew symmetric matrices. Let $$
A = \begin{pmatrix}
0 & 1 & t_1 \\
-1 & 0 & t_2 \\
-t_1 & -t_2 & 0 \\
\end{pmatrix}.
$$

It is easy to see that $\crk(A) = 2$. However, over the free skew field $\K\llangle t_1,t_2 \rrangle$, we can do row and column transformations to transform $A$ to 
$$
\begin{pmatrix}
0 & 1 & 0 \\
-1 & 0 & 0 \\
0 & 0 & [t_2,t_1] \\
\end{pmatrix},
$$
which is clearly of full rank since $[t_2,t_1]=t_2t_1-t_1t_2$ is non-zero over the skew field $\K\llangle t_1,t_2 \rrangle$.
\end{example}

\begin{example}
In \cite{EH}, several low rank examples can be found. For example, they show that  
$$  \begin{pmatrix}
c &d & 0 & 0 \\
0 & 0 & c & d \\
-a & 0 & -b & 0\\
0 & -a & 0 & -b \\
\end{pmatrix} \text{ and } \begin{pmatrix}
-b &-d & 0 & 0 \\
0 & 0 & -c &- d \\
-d & 0 & b & 0\\
c & a & 0 & b \\
\end{pmatrix}$$
are linear matrices that have commutative rank $3$, and non-commutative rank $4$. Here $a,b,c,d$ are variables.
\end{example}

There are very few examples known where there is a discrepancy between commutative and non-commutative rank. One well known family is based on skew symmetric matrices for odd $m$ (see Example~\ref{sk3}). In \cite[Section~4]{DM}, a more interesting family is given based on the Cayley-Hamilton theorem.

\subsection{Linear subspaces of matrices}
Linear matrices can also be studied from the point of view of linear subspaces and their tensor blow-ups. We denote by $\M_{p,q}$ the space of $p \times q$ matrices over the field $\K$.

\begin{definition}
We define the rank of a linear subspace $\mathcal{X} \subseteq \M_{p,q}$ to be the maximal rank among its members, and denote it by $\rm rk(\mathcal{X})$.
\end{definition}
The set of matrices in $\mathcal{X}$ having this maximal rank is Zariski open. Since the underlying field $\K$ is infinite, we can relate the commutative rank of a linear matrix to the rank of a linear subspace (see \cite[Lemma~3.1]{FR}) as follows:

\begin{lemma} [\cite{FR}] \label{crk-eq}
Let $A  = X_0 + t_1X_1 + t_2X_2 + \dots + t_m X_m$ be a linear matrix and let $\mathcal{X} = \spa ( X_0, X_1,X_2, \dots, X_m )$. Then
$$
\crk(A) = \rk(\mathcal{X}).
$$

\end{lemma}
It turns out that non-commutative rank can be understood from the perspective of linear subspaces. In order to do this, we require the notion of tensor blow-ups for linear subspaces.

\begin{definition}
Let $\mathcal{X}$ be a linear subspace of $\M_{k,n}$. We define its $(p,q)$ tensor blow-up $\mathcal{X}^{\{p,q\}}$ to be $$\mathcal{X} \otimes \M_{p,q} = \Big\{ \sum_{i} X_i \otimes T_i \Big|  X_i \in \mathcal{X}, T_i \in \M_{p,q} \Big\}$$ viewed as a linear subspace of $\M_{kp,nq}$. We will write ${\mathcal X}^{\{d\}}={\mathcal X}^{\{d,d\}}$.
\end{definition}

Given $X \in \mathcal{X}$ having rank $r = \rk(\mathcal{X})$, observe that $X \otimes I \in  \mathcal{X}^{\{d\}}$ has rank $rd$. Hence $\rk(\mathcal{X}^{\{d\}})$ is at least $d\cdot\rk(\mathcal{X})$. 

\begin{example}
Let $\mathcal{X}$ denote the linear subspace of skew symmetric $3 \times 3$ matrices. The rank of this subspace is $2$. Let $X_1,X_2,X_3$ be any basis of $\mathcal{X}$. Domokos showed in \cite{Dom00b} that $X_1 \otimes \begin{pmatrix} 1 & 0 \\ 0 & 0 \end{pmatrix} + X_2 \otimes \begin{pmatrix} 0 & 1 \\ 1 & 0 \end{pmatrix} + X_3 \otimes \begin{pmatrix} 0 & 0 \\ 0 & 1 \end{pmatrix} \in \mathcal{X}^{\{2\}}$ has full rank  $6$, which is larger than $2\cdot2 = 4$.
\end{example}

The above example shows $\rk(\mathcal{X}^{\{d\}})$ could be larger than $d\cdot \rk(\mathcal{X})$. However, 
Ivanyos, Qiao and Subrahmanyam showed that there is a very strong restriction on the possible ranks of tensor blow-ups. They proved the following regularity lemma (\cite[Lemma~11 and Remark~10]{IQS}).

\begin{proposition} [\cite{IQS}, Regularity Lemma] \label{regularity}
If $\mathcal{X}$ is a linear subspace of matrices, then $\rk(\mathcal{X}^{\{d\}})$ is a multiple of $d$.
\end{proposition}

In \cite{IQS}, this is proved by giving an algorithm that takes a matrix of rank $\geq rd+1$ in $\mathcal{X}^{\{d\}}$ and produces another matrix in $\mathcal{X}^{\{d\}}$ of rank $\geq (r+1)d$. Analyzing their algorithm (see \cite{IQS,IQS2}), they show that it runs in polynomial time. We give another proof of the regularity lemma using Amitsur's universal division algebras. While our proof is less constructive than the original proof, it is conceptually more satisfying. 

The following characterization of non-commutative rank in terms of ranks of tensor blow-ups appears in \cite{IQS}.

\begin{lemma}[\cite{IQS}] \label{ncrk-eq}
Let $A  = X_0 + t_1X_1 + t_2X_2 + \dots + t_m X_m$ be a linear matrix and let $\mathcal{X} = \spa ( X_0, X_1,X_2, \dots, X_m )$. Then:
$$
\ncrk(A) = \max\limits_d{\displaystyle\frac{\rk(\mathcal{X}^{\{d\}})}{d}} = \lim\limits_{d \to \infty} \displaystyle\frac{\rk(\mathcal{X}^{\{d\}})}{d}.
$$
\end{lemma}

We observe that we can define non-commutative ranks for linear subspaces of matrices, and do so. 
\begin{definition} \label{ncrk-linsub}
For a linear subspace of matrices $\mathcal{X}$, define 
$$\ncrk(\mathcal{X}) = \max{\displaystyle\frac{\rk(\mathcal{X}^{\{d\}})}{d}} = \lim\limits_{d \to \infty} \displaystyle\frac{\rk(\mathcal{X}^{\{d\}})}{d}.$$

\end{definition}

The existence of $\max\limits_d{\displaystyle\frac{\rk(\mathcal{X}^{\{d\}})}{d}}$ and $\lim\limits_{d \to \infty} \displaystyle\frac{\rk(\mathcal{X}^{\{d\}})}{d}$ follows from the following partial increasing property of ranks of blow-ups (see \cite[Corollary~12]{IQS}), along with the aforementioned regularity lemma. 

\begin{lemma} [\cite{IQS}] \label{limit}
Let $\mathcal{X}$ be a linear subspace of matrices, Then for $d \geq n$, $\displaystyle\frac{\rk(\mathcal{X}^{\{d\}})}{d}$ is weakly increasing.
\end{lemma}

Observe that $\ncrk(\mathcal{X}) \geq \rk(\mathcal{X})$, since $\rk(\mathcal{X}^{\{d\}})$ is at least $d\cdot \rk(\mathcal{X})$. On the other hand, it is shown in \cite{FR} that $\ncrk(\mathcal{X}) \leq 2 \rk(\mathcal{X})$. In fact, modifying their argument, one can show that the ratio must be $<2$.

\begin{proposition} \label{<2}
For any linear subspace $\mathcal{X}$, we have $\ncrk(\mathcal{X}) < 2  \rk(\mathcal{X})$. 
\end{proposition}

The authors of \cite{IQS} comment that the statement of Lemma~\ref{limit} is perhaps true for $d < n$ as well, but are unable to prove it. Making use of the proposition above, we are able to show:

\begin{proposition} \label{n/2}
Let $\mathcal{X}$ be a linear subspace of matrices, Then for $d \geq \frac{n}{2} - 1$, $\displaystyle\frac{\rk(\mathcal{X}^{\{d\}})}{d}$ is weakly increasing.
\end{proposition}

We also show that the increasing property need not hold for small values of $d$. We combine a construction of Bergman in \cite{Bergman1} of a rational identity satisfied by $3 \times 3$ matrices but not by $2 \times 2$ matrices, with a construction of Hrube\v s and Wigderson in \cite{HW} to give a counterexample.

\begin{proposition} \label{counterexample}
There exists a linear subspace $\mathcal{X}$ such that 
$$
\frac{\rk(\mathcal{X}^{\{2\}})}{2} > \frac{\rk(\mathcal{X}^{\{3\}})}{3}.
$$
\end{proposition}

In fact, using an existential result in \cite{Bergman1}, we can show:

\begin{theorem} \label{gencounterexample}
For any $n,m \in \Z_{>0}$ such that $n \nmid m$, there is a linear subspace $\mathcal{X}$ such that 
$$
\frac{\rk(\mathcal{X}^{\{n\}})}{n} > \frac{\rk(\mathcal{X}^{\{m\}})}{m}.
$$
\end{theorem}

\subsection{Ratio of non-commutative rank to commutative rank}
We have seen in Proposition~\ref{<2} that the ratio of non-commutative rank to commutative rank is bounded by $2$. In \cite{FR}, Fortin and Reutanauer comment that the bound of $2$ is perhaps not sharp, and suggest that $3/2$ might be the right bound based on the examples available. In this paper, we give a family of examples for which the ratio comes arbitrarily close to $2$, thus showing that the sharpest possible bound is actually $2$! We have seen above that linear matrices can be studied through the linear subspaces of matrices they define, and we give our examples in the language of linear subspaces of matrices. 

Given an element in $v \in \K^n$, we can define a map $L_v: \bigwedge^i\K^n \rightarrow \bigwedge^{i+1}\K^n$ given by $x \mapsto v \wedge x$. This gives a linear map $L: \K^n \rightarrow \Hom(\bigwedge^i\K^n, \bigwedge^{i+1}\K^n)$. The image is a linear subspace. 

\begin{theorem} \label{ratioexample}
Let $\mathcal{X}(p,2p+1)$ denote the image of $L: \K^{2p+1} \rightarrow \Hom(\bigwedge^p \K^{2p+1}, \bigwedge^{p+1}\K^{2p+1})$. We have 
$$
\frac{\ncrk({\mathcal{X}(p,2p+1)})}{\rk({\mathcal{X}(p,2p+1)})} = \frac{2p+1}{p+1}.
$$
\end{theorem}

The linear subspaces in the theorem above provide a family of examples for which the ratio approaches $2$. In fact, these linear subspaces give rise to more examples which have a discrepancy between the commutative and non-commutative rank. 

\begin{corollary} \label{egfamily}
Let $\mathcal{X}(i,n)$ denote the image of the linear map $L: \K^n \rightarrow(\bigwedge^i\K^n,\bigwedge^{i+1}\K^n)$. Then 
\begin{enumerate}
\item $\ncrk(\mathcal{X}(i,n))$ is full;
\item If $i \neq 0,n-1$, then $\rk(\mathcal{X}(i,n))$ is not full. 
\end{enumerate}
\end{corollary}

\subsection{Applications to tensor rank}
Let $\kar K = 0$. A simple dimension counting argument shows that there is an open dense subset containing tensors of border rank $\geq n^3/(3n-2)$. However, the polynomial equations that are satisfied by tensors of border rank $\leq m$ get very complicated as $m$ becomes large. In \cite{Landsberg}, Landsberg gives non-trivial equations for tensors in $\K^m \otimes \K^m \otimes \K^m$ of border rank $2m-3$ when $m$ is even, and $2m-5$ when $m$ is odd.

We show that Landsberg's methods are essentially an investigation of ranks of tensor blow-ups for the linear subspaces in Theorem~\ref{ratioexample}. In \cite{DM}, we showed a concavity property for the ranks of tensor blow-ups. For odd $m$, using the concavity property, we give non-trivial equations for tensors of border rank $2m-4$ in $\K^m \otimes \K^m \otimes \K^m$ (see Theorem~\ref{equations}). We use our equations to prove that certain explicit tensors have border rank $\geq 2m-3$ (see Proposition~\ref{explicit}).

\section{Tensor blow-ups and  Universal division algebras}

\subsection{Tensor blow-ups}
 Let $ A =  X_0 + t_1X_1 + t_2X_2 + \dots + t_{m}X_{m}$ be an $p \times q$ linear matrix. The $(i,j)^{th}$ entry of $A$ is a linear function in the indeterminates $t_k$'s with coefficients in $\K$. In fact if $c_k \in \K$ is the $(i,j)^{th}$ entry of $X_k$, then the $(i,j)^{th}$ entry of $A$ is given by $A_{i,j} =  c_0 + c_1t_1 + \dots + c_mt_m.$ Suppose $S_1,\dots,S_m$ are $d \times d $ matrices, then $X_0 \otimes I + X_1 \otimes S_1 + \dots + X_m \otimes S_m$ is a $p \times q$ block matrix  and the size of each block is $d \times d$. Moreover, the $(i,j)^{th}$ block is $c_0 I + c_1S_1 + \dots +c_mS_m$. 

In effect $X_0 \otimes I + X_1 \otimes S_1 + \dots + X_m \otimes S_m$ is simply the block matrix obtained by substituting $S_k$ for $t_k$ in the linear matrix $A$. Hence, we make the following definition.

\begin{definition}
Let $A = X_0 + t_1X_1 + \dots + t_mX_m$ be a linear matrix. For any $m$-tuple of matrices $S = (S_1,S_2,\dots, S_m)$, we define 
$$
A(S) = X_0 \otimes I + X_1 \otimes S_1 + \dots + X_m \otimes S_m.
$$
\end{definition}

\subsection{The ring of generic matrices}

 Let $\{ t^i_{j,k} | 1\leq j,k \leq d, i \in \Z_{>0}\}$ be a collection of independent commuting indeterminates. For each $i \in \Z_{> 0}$, define the $d \times d$ matrix $T_i = [t^i_{j,k}] $. By a generic matrix, we will refer to a matrix of indeterminates. Let $\K[\{t^i_{j,k}\}]$ denote the polynomial ring in the variables $t_{j,k}^i$ for $1 \leq j,k \leq d$, $i \in \Z_{>0}$. Observe that for each $i$, the generic matrix $T_i$ lies in $\M_{d,d} (\K[\{t^i_{j,k} \}] )$. 
\begin{definition}
The ring of generic matrices $R_d \subseteq \M_{d,d}(\K[\{t^i_{j,k}\}] )$ is defined as the subalgebra generated by $\{T_i| i \in \Z_{> 0}\}$.
\end{definition}

\begin{lemma} \label{eq-lin-blowup rank}
Let $A = X_0 + t_1X_1 + \dots + t_mX_m$ be a linear matrix, and let $\mathcal{X} = \spa (X_1,X_2,\dots,X_m )$. Then, we have
 $$
 \rk(\mathcal{X}^{\{d\}}) = \rk(X_0 \otimes I + X_1 \otimes T_1 + \dots + X_m \otimes T_m),
 $$
 where $T_i$ is a generic matrix for $i  = 1,2,\dots,m$.
\end{lemma}

\begin{proof}
We first show $ \rk(\mathcal{X}^{\{d\}}) \leq \rk(X_0 \otimes I + X_1 \otimes T_1 + \dots + X_m \otimes T_m)$. For $S = (S_0,S_1,\dots,S_m)$ in a non-empty Zariski open subset of $\M_{d,d}^{m+1}$, we have $\rk(X_0 \otimes S_0 + X_1 \otimes S_1 + \dots + X_m \otimes S_m) =  \rk(\mathcal{X}^{\{d\}}) = r$, since $\K$ is infinite. There is an $S$ in this Zariski open subset for which $S_0$ is invertible. For such an $S$, observe that $\rk(X_0 \otimes I + X_1 \otimes S_0^{-1}S_1 + \dots + X_m \otimes S_0^{-1}S_m) = r$. The corresponding $r \times r$ minor in $X_0 \otimes I + X_1 \otimes T_1 + \dots + X_m \otimes T_m$ must be a non-zero polynomial.

The other inequality, i.e.,  $\rk(\mathcal{X}^{\{d\}}) \geq \rk(X_0 \otimes I + X_1 \otimes T_1 + \dots + X_m \otimes T_m)$ is straightforward.

\end{proof}

\begin{example} \label{eg-regularity}
Let $A = X_0 + t_1 X_1 + t_2X_2 =  \left[ \arraycolsep=5pt\def\arraystretch{1} \begin{array}{ccc}
0 & 1 & t_1 \\
-1 & 0 & t_2 \\
-t_1 & -t_2 & 0 \\
\end{array} \right].$ Then for generic matrices $T_1,T_2$, we have:
 $$
 A(T_1,T_2) = X_0 \otimes I + X_1 \otimes T_1 + X_2 \otimes T_2 =  \left[ \arraycolsep=5pt\def\arraystretch{1} \begin{array}{ccc}
0 & I & T_1 \\
-I & 0 & T_2 \\
-T_1 & -T_2 & 0 \\
\end{array} \right]. 
$$

As observed in the introduction, we can do row and column transformations to transform  $$ 
\left[ \arraycolsep=5pt\def\arraystretch{1} \begin{array}{ccc}
0 & I & T_1 \\
-I & 0 & T_2 \\
-T_1 & -T_2 & 0 \\
\end{array} \right] \longrightarrow \left[ \arraycolsep=5pt\def\arraystretch{1} \begin{array}{ccc}
0 & I & 0 \\
-I & 0 & 0 \\
0 & 0 & [T_2,T_1] \\
\end{array} \right].
$$

Hence the $\rk A(T_1,T_2) = 2d + \rk([T_1,T_2])$. If the $T_i$ are generic matrices of size $1 \times 1$, then $[T_1,T_2] = 0$, and if $T_1,T_2$ are generic matrices of size $d \times d$ for $d \geq 2$, then $[T_1,T_2]$ is invertible, and hence of full rank. Thus for $T_1,T_2$ generic matrices of size $d \times d$, we have $$
\rk A(T_1,T_2) = \begin{cases}
2 & if\  d = 1,\\
3d & if\  d \geq 2.\\
\end{cases}
$$

In particular, observe that $\rk A(T_1,T_2)$ is always a multiple of $d$. Using Lemma~\ref{eq-lin-blowup rank}, one sees that the regularity lemma is satisfied in for the linear subspace of $3 \times 3$ skew symmetric matrices. 
\end{example}

\subsection{Universal division algebras}
Observe, as in Example~\ref{eg-regularity}, that for generic $d \times d$ matrices, the expression $[T_1,T_2]$ was either identically zero, or invertible depending upon the value of $d$. This is a special case of a surprising general phenomenon, namely that any non-zero non-commutative rational expression in some $d \times d$ generic matrices must in fact be invertible! This follows from the fact that Amitsur's universal division algebras are division algebras. We describe these universal division algebras.

Recall the ring of generic matrices $R_d \subseteq \M_{d,d}(\K[\{t^i_{j,k}\}]).$ Let $Z_d$ denote the center of $R_d$, and let its field of fractions be $Q_d$. The following result is due to Amitsur (see \cite{Amitsur,Amitsur2,Amitsur3}). One can also find it in standard texts (for example \cite[Section~7.2]{Cohn:Skewfields}).

\begin{theorem} [Amitsur]
$\operatorname{UD}(d) := Q_d \otimes_{Z_d} R_d$ is a division algebra and is called a universal division algebra of degree $d$. 
\end{theorem}

\begin{proof}
Posner proved that the central quotient of a prime PI-ring is a simple algebra (see \cite{Po}). The ring $R_d$ satisfies a polynomial identity, namely the Amitsur-Levitzki polynomial. Amitsur showed (see \cite[Theorem~4]{Amitsur}) that $R_d$ is in fact an integral domain, and in particular a prime ring. Hence its central quotient $\operatorname{UD}(d)$ is a simple algebra. By the Wedderburn-Artin theorem (see \cite[Section~3.13]{Jac-ba2}), it must be a matrix algebra over a division algebra, i.e., $\operatorname{UD}(d) \cong \M_{r,r}(D)$ for some division algebra $D$ and $r\in \Z_{>0}$. Further, since $R_d$ is an integral domain, $\operatorname{UD}(d)$ has no nilpotents. Hence $\operatorname{UD}(d) \cong \M_{1,1}(D) \cong D$.
\end{proof}

Note that $\operatorname{UD}(d) \subseteq \M_{d,d}(\K(\{t^i_{j,k}\}))$. We now give another proof of the regularity lemma, as we mentioned in the introduction. 
 
\begin{proof}[Proof of Theorem~\ref{regularity}]
Let $X_0,X_1,X_2,\dots,X_m$ span the linear subspace $\mathcal{X} \subseteq \M_{p,q}$, and set $A = X_0 + t_1X_1 + \dots + t_mX_m$. Then by Lemma~\ref{eq-lin-blowup rank}, we have $$\rk(\mathcal{X}^{\{d\}}) = \rk(X_0 \otimes I + X_1 \otimes T_1 + \dots + X_m \otimes T_m).$$

 $A(T_1,T_2,\dots,T_m) = X_0 \otimes I + X_1 \otimes T_1 + \dots + X_m \otimes T_m$ can be viewed as a $p \times q$ block matrix whose blocks are linear expressions in the generic matrices $T_i$, and in particular elements of $\operatorname{UD}(d)$, a division algebra. By row and column operations in $\operatorname{UD}(d) \subseteq \M_{d,d}(\K(\{t^i_{j,k}\}))$, we can make the transformation: 
$$(X_0 \otimes I + X_1 \otimes T_1 + \dots + X_m \otimes T_m) \longrightarrow 
\sbox0{$\begin{matrix}I & & \\& \ddots &\\ & &I \end{matrix}$}
\left[
\begin{array}{c|c}
\usebox{0}&\makebox[\wd0]{\huge $0$}\\
\hline
  \vphantom{\usebox{0}}\makebox[\wd0]{\huge $0$}&\makebox[\wd0]{\huge $0$}
\end{array}
\right]
$$

Since each $I$ denotes a $d \times d$ identity matrix, it is clear that $\rk(A(T_1,\dots,T_m)) = \rk(\mathcal{X}^{\{d\}})$ is a multiple of $d$. 
\end{proof}

\section{Combinatorics of ranks of tensor blow-ups}

\subsection{Weakly increasing property of blow-ups}
We modify the argument used by Fortin and Reutenauer (see~\cite{FR}) to prove Proposition~\ref{<2}.

\begin{proof} [Proof of Proposition~\ref{<2}]
Let $r$ be the smallest non-negative integer such that we have a linear subspace $\mathcal{X} \subseteq \M_{p,q}$ of rank $r$ for some $p,q$, such that $\ncrk(\mathcal{X}) = 2r$. We have $r > 1$ since $\rk(\mathcal{X}) = 1$ implies $\ncrk(\mathcal{X}) = 1$ (see \cite[Remark~1]{FR} and \cite[Lemma~2.9]{DM}). 

We use a result of Flanders (see \cite[Lemma~1]{Flanders}) to see that $\mathcal{X}$ is equivalent to a subspace of the form $\begin{pmatrix} A & 0 \\ C & B \end{pmatrix}$ with $C$ of size $r \times r$ (see also \cite[Corollary~2]{FR}). Since $\ncrk(\mathcal{X}) = 2r$, we must have $\ncrk(A) \geq r$, since we must have at least $2r$ linearly independent rows. But $A$ has only $r$ columns, and hence $\ncrk(A) = r$. A similar argument considering columns shows that $\ncrk(B) = r$. 

We have $\rk(A),\rk(B) \geq r/2$ because the ratio is at most $2$. We cannot have $\rk(A) = r/2$ or $\rk(B) = r/2$ as that would violate the minimality of $r$. Thus $\rk(A), \rk(B) > r/2$. However, this means that $\rk(\mathcal{X}) \geq \rk(A) + \rk(B) > r$. 

\end{proof}

To prove Proposition~\ref{n/2}, we improve the proof of Proposition~\ref{limit} given in \cite{IQS} by making use of Proposition~\ref{<2}.

\begin{proof}[Proof of Proposition~\ref{n/2}]
Suppose $\rk(\mathcal{X}^{\{d\}})/d = r$. Choose a basis $X_1,\dots,X_m$ of ${\mathcal X}$.
There exist $T_1,\dots,T_m\in \M_{d,d}$ such that
 $\sum_i X_i \otimes T_i \in \mathcal{X}^{\{d\}}$ has rank $rd$. Choose $a_1,\dots,a_m\in K$  such that $\sum_i a_i X_i \in \mathcal{X}$ has rank equal to $\rk(\mathcal{X})$.

Then let $\widetilde{T_i} \in \M_{d+1,d+1}$ be given by $$\widetilde{T_i} = \left[\arraycolsep=5pt\def\arraystretch{1} \begin{array}{c|c}
\Scale[1.5]{T_i} & \begin{array}{c} 0 \\ \vdots \\ 0 \\ \end{array} \\
\hline
\begin{array}{ccc} 0 & \dots & 0 \end{array} & a_i \\
\end{array} \right].$$

Then it is easy to see that 
$$
\textstyle \rk(\sum_{i=1}^m X_i \otimes \widetilde{T_i})\geq \rk(\sum_{i=1}^m X_i \otimes T_i) + \rk(\sum_{i=1}^m a_iX_i)
= rd + \rk(\mathcal{X})
$$

Furthermore, we have 
$$\textstyle \rk(\mathcal{X}) > \frac{1}{2} \ncrk(\mathcal{X}) \geq \frac{1}{2} r.$$

In the above, the first inequality follows from Proposition~\ref{<2}, and the second follows from the Definition~\ref{ncrk-linsub}. Hence, we have 
$$\textstyle \rk(\mathcal{X}^{\{d+1\}}) \geq \rk(\sum_i X_i \otimes \widetilde{T_i}) > rd + \frac{1}{2} r.$$

Since $d \geq \frac{n}{2} - 1 \geq \frac{r}{2} - 1$, we have $\rk({\mathcal X}^{\{d+1\}})>rd + \frac{1}{2}r \geq (r-1)(d+1)$. Now, by the regularity lemma (Proposition~\ref{regularity}) we must have $\rk(\mathcal{X}^{\{d+1\}})/(d+1) \geq r =\rk(\mathcal{X}^{\{d\}})/d$.
\end{proof}

The following result follows from the concavity of tensor blow-ups (see \cite[Proposition~2.10]{DM}).

\begin{proposition}
Let $n \geq 2$ and suppose $d \geq n-1$, and assume $\rk(\mathcal{X}^{\{d+1\}})/(d+1) = r$. Then $\rk(\mathcal{X}^{\{d\}})/d \geq r$. 
\end{proposition}

Combining Proposition~\ref{n/2} with the above proposition, we get 

\begin{corollary}
Let $n \geq 2$. Then $\rk(\mathcal{X}^{\{d\}})/d$ is constant for $d \geq n-1$. 
\end{corollary}

\subsection{Rational identities}
In \cite{Bergman1,Bergman2}, Bergman proved a number of remarkable results on rational relations and rational identities in division rings. In particular, he came up with an explicit construction of a rational expression which is an identity on $3 \times 3$ matrices, but invertible on general $2 \times 2$ matrices. We introduce some notation. Let $Y'$ denote the commutator $[X,Y]$, and let $\delta(Y)$ denote $(Y^2)'[(Y^{-1})']^{-1}$. In \cite{Bergman1}, Bergman proves the following result (see also \cite[Theorem~7.4.3]{Cohn:Skewfields}).

\begin{theorem} [\cite{Bergman1}] \label{Bergman}
Let $n = 2$ or $3$. For $X,Y \in \M_{n,n}(\K)$, we have:
$$
\psi = \delta(Y')\delta(Y'')[(\delta(Y'')^{-1})'][(\delta(Y''')^{-1})'] = \begin{cases}
1 &  if\  n=3, \\
0 & if\  n = 2.
\end{cases}
$$
\end{theorem}

\begin{corollary}
The rational expression $\psi - 1$ is an identity for $3 \times 3$ matrices, but is invertible for general choices of $2 \times 2$ matrices. 
\end{corollary}

Bergman also showed the existence of such rational functions more generally. Let $\mathcal{E}(d)$ be the set of rational expressions that can be evaluated on generic $d \times d$ matrices. 

\begin{theorem} [\cite{Bergman1}] \label{gen.rat}
Assume $n,m \in \Z_{>0}$. Then $\mathcal{E}(n) \subseteq \mathcal{E}(m)$ if and only if 
$n \mid m$.
\end{theorem}

\subsection{Non-commutative arithmetic circuits with division}
A non-commutative arithmetic circuit is a directed acyclic graph, whose vertices are called gates. Gates of in-degree $0$ are elements of $\K$ or variables $t_i$. The other allowed gates are inverse, addition and multiplication gates of in-degrees 1, 2 and 2 respectively. The edges going into an multiplication gate are labelled left and right to indicate the order of multiplication. 
A formula is a circuit, where every node has out-degree at most $1$. The number of gates in a circuit is called its size. 
Let $\Phi$ be a circuit in $m$ variables. It is easy to observe that each output gate of a circuit $\Phi$ computes a rational expression. 
We denote by $\widehat\Phi(T)$ the evaluation of $\Phi$ at $T = (T_1,T_2,\dots,T_m) \in \M_{p,p}^m$. In the process of evaluation, if the input of an inverse gate is not invertible, then $\widehat{\Phi}(T)$ is undefined. $\Phi$ is called a correct circuit if $\widehat\Phi(T)$ is defined for some $T$. For further details, we refer to \cite{HW}.

In \cite{HW}, Hrube\v s and Wigderson reduce non-commutative rational identity testing to deciding the invertibility of linear matrices. A deterministic algorithm over $\Q$ for deciding the invertibility of linear matrices was given by Garg, Gurvits, Oliviera and Wigderson in \cite{GGOW} by analyzing an algorithm of Gurvits in \cite{Gurvits}. In \cite{DM}, we give bounds for the size of matrices required to detect invertibility, and this gives another proof that over $\Q$, invertibility of linear matrices can be decided in polynomial time. Moreover, the bounds in \cite{DM} immediately show that invertibility of a linear matrix can be decided in randomized polynomial time over arbitrary characteristic. In \cite{IQS2}, Ivanyos, Qiao and Subrahmanyam use the bounds in \cite{DM} to give a deterministic algorithm that works over arbitrary characteristic.

Given a non-commutative formula of size $n$, Hrube\v s and Wigderson construct a family of linear matrices $A_u$ for each gate $u$ of the formula. We refer to \cite[Theorem~2.5]{HW} for details. We are content to remark that these matrices can be constructed explicitly in time which is polynomial in $n$. We recall \cite[Propostion~7.1]{HW}.

\begin{proposition} [\cite{HW}] \label{circuit}
Let $R$ be a ring which contains $\K$ in its center. For a formula $\Phi$, and $a_1,a_2,\dots,a_m \in R$, the following are equivalent:
\begin{enumerate}
\item $\widehat\Phi(a_1,a_2,\dots,a_m)$ is defined.
\item For every gate $u$, the $A_u(a_1,a_2,\dots,a_m)$ is invertible. 
\end{enumerate}
\end{proposition}

Now, we can put Bergman's results together with Hrube\v s and Wigderson's results to give a proof of Proposition~\ref{counterexample}.

\begin{proof}[Proof of Proposition~\ref{counterexample}]
Let $\Phi$ be the non-commutative formula that computes the rational expression $(\psi - 1)^{-1}$. By the construction of Hrube\v s and Wigderson mentioned above, we have linear matrices $A_u$ for each gate $u$. Observe that  $\widehat\Phi(T)$ is defined for $T = (T_1,T_2, \dots, T_m)$ where the $T_i$ are generic $2 \times 2$ matrices by Theorem~\ref{Bergman}.  Thus, the $A_u(T)$ is invertible for all $u$. 

On the other hand, if the $T_i$ are generic $3 \times 3$ matrices, then once again by Theorem~\ref{Bergman}, $\widehat\Phi(T)$ is not defined. Thus, for some $u$, $A_u$ is not invertible. For this $u$, write $A_u = X_0 + t_1X_1 + \dots + t_mX_m$ and let $\mathcal{X} = \spa ( X_0,X_1,\dots, X_m )$. Then, using Lemma~\ref{eq-lin-blowup rank}, we conclude $$
\frac{\rk(\mathcal{X}^{\{2\}})}{2} > \frac{\rk(\mathcal{X}^{\{3\}})}{3}.
$$

\end{proof}

For the general case, we use Theorem~\ref{gen.rat}.

\begin{proof} [Proof of Theorem~\ref{gencounterexample}]
If $n \nmid m$, then there exists $r \in \mathcal{E}(n)$ such that $r \notin \mathcal{E}(m)$. Let $\Phi$ be the non-commutative formula that computes $r$. The argument in the proof of Proposition~\ref{counterexample} applied to $\Phi$ gives the required conclusion. 
\end{proof}

\section{Ratio of non-commutative and commutative ranks}
We assume $\kar K =0$ for this section. For $v \in \K^n$, define $L_v : \bigwedge^p(\K^n) \rightarrow \bigwedge^{p+1}(\K^n)$ by $w \mapsto v \wedge w$. Let $e_1,e_2,\dots,e_n$ be a basis for $\K^n$. Note  that a basis for $\bigwedge^p(\K^n)$ is given by $\{e_{i_1} \wedge e_{i_2} \wedge \dots \wedge e_{i_p} | 1 \leq i_1 < i_2 < \dots < i_p \leq n \}$. 

 Let $A(p,n)$ denote the linear matrix given by $t_1L_{e_1} + t_2 L_{e_2} + \dots + t_nL_{e_n}$. 
 
 \begin{lemma} \label{basis choice}
 For a particular choice of basis, the linear matrix $A(p,n)$ has the form 
 $$
\left[ \arraycolsep=5pt\def\arraystretch{2} \begin{array}{c|c}
t_n I & A(p-1,n-1) \\
\hline
A(p,n-1) & \Scale[1.5]{0} 
\end{array} \right] 
$$
 \end{lemma}

\begin{proof}
Let  
\begin{align*}
A & = \{  (e_{i_1} \wedge e_{i_2} \wedge \dots \wedge e_{i_{p-1}})\wedge e_n \mid 1 \leq i_1 < \dots < i_{p-1} \leq n-1 \}, \text{ and} \\
 B & = \{e_{i_1} \wedge e_{i_2} \wedge \dots \wedge e_{i_p} \mid 1 \leq i_1 < \dots < i_p \leq n-1 \}.
\end{align*}

Then clearly $A \cup B$ is a basis for $\bigwedge^p(\K^n)$. Similarly, let 
\begin{align*}
C & = \{  (e_{i_1} \wedge e_{i_2} \wedge \dots \wedge e_{i_{p}})\wedge e_n \mid  1 \leq i_1 < \dots < i_{p} \leq n-1 \}, \text{ and} \\
D & = \{e_{i_1} \wedge e_{i_2} \wedge \dots \wedge e_{i_{p+1}} \mid 1 \leq i_1 < \dots < i_{p+1} \leq n-1 \}.
\end{align*}
Then $C \cup D$ is a basis for $\bigwedge^{p+1}(\K^n)$. It is easy to see that there $L_{e_n}: B \rightarrow C$ is a bijection. Now, order the basis elements for $\bigwedge^p(\K^n)$ by taking the basis vectors from $B$ first, and then from $A$. For $\bigwedge^{p+1}(\K^n)$, order the basis vectors by taking the basis vectors from $C$ first, and then from $D$. Within the basis vectors of $C$, we order them in the same order as the vectors from $B$ via the aformentioned bijection given by $L_{e_1}$. 
\end{proof}

\begin{remark}
The description in \cite[Section~4]{Landsberg-1} is the same as the one above. See also \cite[Section~2]{MR,MRC}.
\end{remark}

\begin{corollary} \label{ind.fullrank}
If $A(p,n)$ has full column rank, then so does $A(p-1,n-1)$. Similarly, if $A(p,n)$ has full row rank, then so does $A(p,n-1)$.
\end{corollary}

\begin{corollary} \label{crkLv}
For any non-zero $v \in \K^n$, $\rk(L_v) = {n-1 \choose p}.$
\end{corollary}

\begin{proof}
Assume without loss of generality that $v = e_n$. By the above choice of basis $L_{e_n} = \left[ \arraycolsep=5pt\def\arraystretch{1} \begin{array}{c|c}
I & 0  \\
\hline
0  & 0 
\end{array} \right] $. Hence $$
\rk(L_v) = |B| = |C| = {n-1 \choose p}.
$$
\end{proof}

\begin{corollary} \label{crkAp}
We have $\crk(A(p,n)) = \rk(\mathcal{X}(p,n)) =  {n-1 \choose p}$. 
\end{corollary}

\begin{proof}
This follows from Lemma~\ref{crk-eq}.
\end{proof}

Observe that in order to prove Theorem~\ref{ratioexample}, it suffices to prove that $\ncrk(A(p,2p+1)) = \frac{2p+1}{p+1} {2p \choose p} = {2p+1 \choose p}$. Further note that $\dim \bigwedge^p(\K^{2p+1}) = \dim \bigwedge^{p+1}(\K^{2p+1}) = {2p+1\choose p}$. In other words, we want to show that the non-commutative rank of $A(p,2p+1)$ is full. We will use a result of Landsberg in \cite{Landsberg}.

\begin{proposition} [\cite{Landsberg}] \label{toeplitz}
Let $e_1,e_2,\dots,e_{2p+1}$ be a basis for $\C^{2p+1}$. For $-p \leq r \leq p$, let $S_r$ be the $(p+1) \times (p+1)$ matrix such that 
$$S_r (j,k) = \begin{cases} 
1 & if\ j= k + r, \\
0 & otherwise. 
\end{cases}.$$

Then, we have that $ L_{e_1} \otimes S_{-p} + L_{e_2} \otimes S_{-p+1} + \dots + L_{e_{2p+1}} \otimes S_{p}$ is invertible. 
\end{proposition}

The set $\{S_{-p},S_{-p+1},\dots,S_{p}\}$ is a basis for the space of  Toeplitz matrices of size $(p+1) \times (p+1)$. Any other basis of the Toeplitz matrices would work as well. 

\begin{corollary} \label{ncrkAp}
The linear matrix $A(p,2p+1)$ has full non-commutative rank.
\end{corollary}

\begin{proof} [Proof of Theorem~\ref{ratioexample}]
Observe that the linear subspace associated to $A(p,2p+1)$ is $\mathcal{X}(p,2p+1)$. Hence Corollary~\ref{crkAp} and Corollary~\ref{ncrkAp} give us the required conclusion.
\end{proof}

\begin{proof} [Proof of Corollary~\ref{egfamily}]
To prove $(1)$, consider $A(i,n)$. If $i< n/2$, then let $k = n - 2i - 1$. The linear matrix $A(i+k,n+k)$ has full column rank by Corollary~\ref{ncrkAp}, since $2(i+k) + 1 = n+k$. By repeated application of Corollary~\ref{ind.fullrank}, we conclude that $A(i,n)$ has full column rank. Since $i < n/2$, the matrix $A(i,n)$ has more rows than columns, and hence has full non-commutative rank.

If $i \geq n/2$, then we observe that $A(i,2i+1)$ has full non-commutative rank. Once again by repeated application of Corollary~\ref{ind.fullrank}, we conclude that $A(i,n)$ has full row rank. Since $A(i,n)$ has more columns than rows, it has full non-commutative rank. 

Finally, observe that the linear subspace defined by $A(i,n)$ is the linear subspace $\mathcal{X}(i,n)$.

To prove $(2)$, use Corollary~\ref{crkAp}.
\end{proof}

\section{Lower bounds on border rank}

\begin{definition}
For a tensor $T \in \K^{a_1} \otimes \K^{a_2} \otimes \dots \otimes \K^{a_l}$, we define its tensor rank $\trk(T)$ as the smallest $m$ such that $T$ can be written as a sum of $m$ pure tensors.
\end{definition}
 Let $Z_m$ denote the set of tensors with tensor rank $\leq m$. The set $Z_m$ need not be a Zariski closed subset, and we can consider its Zariski closure $\overline{Z_m}$. 

\begin{definition}
For a tensor $T \in \K^{a_1} \otimes \K^{a_2} \otimes \dots \otimes \K^{a_l}$, we define its border rank $\brk(T)$ as the smallest $m$ such that $T \in \overline{Z_m}$.
\end{definition} 

Tensor rank and border rank have been studied extensively, especially with respect to the matrix multiplication tensor. See \cite{Landsbergtext} for details.

We consider tensor product spaces with three tensor factors. Given a tensor in $T \in \K^{a} \otimes \K^{b} \otimes \K^{c}$, we can write $T = \sum_i s_i \otimes X_i$, with $s_i \in \K^a$ and $X_i \in \K^b \otimes \K^c$. Let $L: \K^a \rightarrow \M_{p,q}$ be a linear map, and denote the image by $\mathcal{X}_L$. We identify $\K^b \otimes \K^c$ with $\M_{b,c}$ and identify $\M_{p,q} \otimes \M_{b,c}$ with $\M_{pb,qc}$. This gives the following map.  

\begin{align*}
\psi_L: \K^a \otimes \K^b \otimes \K^c & \longrightarrow \M_{pb,qc}\\
\sum_i s_i \otimes X_i & \longmapsto \sum_i L(s_i) \otimes X_i.
\end{align*}

\begin{lemma}
For a tensor $T \in \K^a \otimes \K^b \otimes \K^c$ we have $\rk(\psi_L(T)) \leq \brk(T)\rk(\mathcal{X}_L)$. 
\end{lemma}

\begin{proof}
Let $T = s \otimes b \otimes c$ be a tensor of rank $1$. Then $\psi_L(T) = L(a) \otimes (b \otimes c)$, and hence $\rk(\psi_L(T)) \leq \rk(L(a)) \leq \rk(\mathcal{X}_L)$. Therefore,  if we take a tensor $T \in \K^a \otimes \K^b \otimes \K^c$ of rank $r$, then $\rk(\psi_L(T)) \leq r\rk(\mathcal{X}_L) = D$. Observe that the $(D+1) \times (D+1)$ minors of $\psi_L(T)$ are polynomial equations that vanish all tensors of rank $\leq r$, i.e they vanish on $Z_r$. Hence these equations vanish on $\overline{Z_r}$ as well. 

Hence if $\brk(T) = r$, we must have $\rk(\psi_L(T)) \leq D = r \rk(\mathcal{X}_L) = \brk(T) \rk(\mathcal{X}_L)$.
\end{proof}

\begin{remark}
In particular, $\displaystyle \frac{\rk(\psi_L(T))}{\rk(\mathcal{X}_L)}$ is a lower bound for $\brk(T)$. Further, observe that $\psi_L(T) \in \mathcal{X}_L^{\{p,q\}}$, and hence $\rk(\psi_L(T)) \leq \rk(\mathcal{X}_L^{\{b,c\}})$. Hence in order to get a good lower bound, it would be useful for the blow-up to have large rank, which in turn is only possible if $\mathcal{X}_L$ has a large ratio of non-commutative rank to commutative rank.
\end{remark}

\begin{corollary} \label{D+1}
Let $D = r \rk(\mathcal{X}_L)$. Then the $(D+1) \times (D+1)$ minors of $\psi_L(T)$ give equations that are satisfied by all tensors of border rank $\leq r$. 
\end{corollary}

Landsberg's technique (see \cite{Landsberg}) for obtaining lower bounds for border rank is the same as the one we describe above. For any $r$, the above corollary gives polynomials that are satisfied by all tensors having border rank $\leq r$. It follows that if these polynomials do not vanish on a tensor $T$, then we must have $\brk(T) >r$, providing a possible method for showing lower bounds for border rank. However, this method is only useful if these polynomials are non-trivial, i.e., not identically zero. The non-triviality of these equations essentially depends on the rank of the blow-up $\mathcal{X}_L^{\{b,c\}}$.

\begin{lemma} \label{non-triv}
One of the $d \times d$ minors of $\psi_L(T)$ is a non-trivial polynomial if and only if $\rk(\mathcal{X}_L^{\{b,c\}}) \geq d$.
\end{lemma}

\begin{proof}
Suppose $\rk(\mathcal{X}_L^{\{b,c\}}) \geq d$. Since $\im(\psi_L) = \mathcal{X}_L^{\{b,c\}}$, there exists $T_1 \in \K^a \otimes \K^b \otimes \K^c$ such that $\rk(\psi_L(T_1)) = \rk(\mathcal{X}_L^{\{b,c\}}) \geq d$. Hence there is a $d \times d$ minor in $\psi_L(T_1)$ that is non-zero, and hence that $d \times d$ minor is a non-trivial polynomial. 

The converse follows immediately since the underlying field $\K$ is infinite. 
\end{proof}

We discuss a few results that can help in finding lower bounds for the ranks of blow-ups.

\begin{lemma} \label{d-kd}
For a linear subspace $\mathcal{X}$, if $\mathcal{X}^{\{d\}}$ has full rank, then $\mathcal{X}^{\{kd\}}$ has full rank for any $k \in \Z_{>0}$. 
\end{lemma}

\begin{proof}
If $\rk(\mathcal{X}^{\{d\}})$ is full, then we have some $\sum_i X_i \otimes S_i$ is invertible for some $X_i \in \mathcal{X}$ and $S_i \in \M_{d,d}$. Clearly, $(\sum_i X_i \otimes S_i) \otimes I_k$ is also invertible, where $I_k \in \M_{k,k}$ denotes the identity matrix. Thus $\sum_i X_i \otimes(S_i \otimes I_k) \in \mathcal{X}^{\{kd\}}$ is invertible.  
\end{proof}

In view of Corollary~\ref{D+1} and Lemma~\ref{non-triv}, it would be useful to show lower bounds on the rank of blow-ups. For this, the concavity properties of blow-ups that we showed in \cite{DM} will be very useful. 

\begin{proposition} [\cite{DM}] \label{concavity}
For a linear subspace of matrices $\mathcal{X}$, let $r(p,q) = \rk(\mathcal{X}^{\{p,q\}})$. Then we have:
\begin{enumerate}
\item $r(p,q+1) \geq r(p,q)$;
\item $r(p+1,q)\geq r(p,q)$;
\item $r(p,q+1) \geq \frac{1}{2}(r(p,q) + r(p,q+2))$;
\item $r(p+1,q)\geq \frac{1}{2}(r(p,q) + r(p+2,q))$.
\end{enumerate}
\end{proposition}

In particular, this shows that that $r(p,q)$ is increasing and concave down in either variable independently.

\section{Border rank of tensors in $\K^m \otimes \K^m \otimes \K^m$ for odd $m$}
Let $m = 2p+1$ be a positive odd integer. Let $L : \K^m \rightarrow \Hom(\bigwedge^{p} \K^m ,\bigwedge^{p+1}\K^m)$ be the linear map defined in Theorem~\ref{ratioexample}. For, this $L$, we define $\psi_L$ as in the previous section, i.e., 
\begin{align*}
\psi_L : \K^m \otimes \K^m \otimes \K^m &\longrightarrow \M_{{2p+1\choose p}m,{2p+1\choose p}m} \\
\sum_i s_i \otimes X_i & \mapsto \sum_i L(s_i) \otimes X_i.
\end{align*}

\begin{theorem} \label{equations}
Let $\psi_L$ be as above, and let $D = {2p\choose p}(2m-4)$. Then at least one of the $(D +1) \times (D+1)$ minors of $\psi_L$ gives a non-trivial equation for tensors in $\K^m \otimes \K^m \otimes \K^m$ of border rank $\leq 2m-4$.  
\end{theorem}

\begin{proof} 
Observe that $\rk(\mathcal{X}_L) = {2p \choose p}$ by Corollary~\ref{crkAp}. Hence, by Corollary~\ref{D+1} and Lemma~\ref{non-triv}, it suffices to show that $\rk(\mathcal{X}_L^{\{m\}}) \geq D + 1$. 

By Proposition~\ref{toeplitz}, we know that $\mathcal{X}_L^{\{p+1\}}$ has full rank and as a consequence of Lemma~\ref{d-kd}, we have $\mathcal{X}_L^{\{2p+2\}}$ has full rank as well. To find lower bounds on $\rk(\mathcal{X}_L^{\{2p+1\}})$, we use the properties from Proposition~\ref{concavity}.

Let $M = \dim\bigwedge^{p} \K^m = \dim \bigwedge^{p+1}\K^m = {2p+1 \choose p}$, and let $r(p,q) = \rk(\mathcal{X}^{\{p,q\}}).$ Then we have $r(p+1,p+1) = (p+1)M$, and $r(2p+2,2p+2) = (2p+2)M$ by the above discussion. We have $$
r(p+1,2p+1) \geq r(p+1,p+1) \geq (p+1)M.
$$
 Further, by concavity in the second variable, we have 
\begin{align*}r(2p+2,2p+1) &\geq \frac{(2p+1)r(2p+2,2p+2) + r(2p+2,0)}{2p+2}\\
 & \geq \frac{(2p+1)(2p+2)M}{2p+2} \\
 & = (2p+1)M.
\end{align*} 

Now, by concavity in the first variable, we have 
\begin{align*}
r(2p+1,2p+1) &\geq \frac{pr(2p+2,2p+1) + r(p+1,2p+1)}{p+1} \\
& \geq \frac{p(2p+1)M + (p+1)M}{p+1} \\
&= \frac{2p^2 + 2p + 1}{p+1} M. 
\end{align*}

Hence, we have 
\begin{align*}
\frac{\rk(\mathcal{X}_L^{\{2p+1\}})}{{2p \choose p}} & \geq \frac{(2p^2 + 2p + 1){2p+1\choose p}}{(p+1) {2p \choose p}}\\
& = \frac{(2p^2 + 2p + 1)(2p+1)}{(p+1) (p+1)} \\
& > 4p - 2 \\
& = 2m-4
\end{align*} 

Thus $\rk(\mathcal{X}_L^{\{m\}}) > {2p \choose p} (2m-4)$ as required. 

\end{proof}

Recall that the non-commutative rank is at most twice the commutative rank. Hence
$$
\displaystyle \frac{\rk(\mathcal{X}_L^{\{m\}})}{\crk(\mathcal{X}_L)} \leq \frac{ m \cdot \ncrk(\mathcal{X}_L)}{\crk(\mathcal{X}_L)} < 2m.
$$ 
This alone shows that there is very little room for improvement for the lower bounds we obtain using this method.

\begin{remark}
For $m = 5$ i.e., $p = 2$, Landsberg shows that in fact $\mathcal{X}_L^{\{m\}}$ has full rank, thus giving non-trivial equations for tensors of border rank $8$. Experimental evidence shows that in fact this is true for $p = 3$ and $4$ as well, suggesting that it is perhaps true for all $p$, which would give non-trivial equations for tensors of border rank $2m-2$. 
\end{remark}

In $\K^m \otimes \K^m \otimes \K^m$, Landsberg gives explicit tensors having border rank $\geq 2m-2$ (resp. $ 2m-4$) when $m$ is even (resp. odd) (see \cite{Landsberg}). For $m$ odd, we can give explicit tensors whose border rank is $\geq 2m-3$. 

Let $m = 2p+1$ be odd, and let $S_r \in \M_{p+1,p+1}$ for $-p \leq r \leq p$ be as in Proposition~\ref{toeplitz}. For each $r$, consider $Q_r = S_r \oplus S_r \in \M_{2p+2,2p+2}$, and let $\widetilde{Q}_r \in \M_{2p+1,2p+1}$ be the matrix obtained from $Q_r$ by removing the last column and last row of $Q_r$. Identifying $\M_{2p+1,2p+1}$ with $\K^m \otimes \K^m$, we can consider the tensor $T = \sum\limits_{i = 1}^{m} e_i \otimes \widetilde{Q}_{i-p-1} \in \K^m \otimes \K^m \otimes \K^m$, where $e_1,e_2,\dots,e_m$ is the standard basis for $\K^m$.

\begin{proposition} \label{explicit}
The tensor $T = \sum\limits_{i = 1}^{m} e_i \otimes \widetilde{Q}_{i-p-1} \in \K^m \otimes \K^m \otimes \K^m$ has border rank $\geq 2m-3$. 
\end{proposition}

\begin{proof}
Let $L : \K^m \rightarrow \Hom(\bigwedge^{p} \K^m ,\bigwedge^{p+1}\K^m)$ be the linear map defined in Theorem~\ref{ratioexample}. We have $\psi_L(T) = \sum\limits_{i = 1}^{m} L(e_i) \otimes \widetilde{Q}_{i-p-1} \in \M_{{2p+1 \choose p}m, {2p+1\choose p}m}$. Observe that $A = \sum\limits_{i = 1}^{m} L(e_i) \otimes {Q}_{i-p-1} \in \M_{{2p+1 \choose p}(m+1), {2p+1\choose p}(m+1)}$ has full rank by Proposition~\ref{toeplitz} and Lemma~\ref{d-kd}. Observe that $\psi_L(T)$ is obtained by removing ${2p+1 \choose p}$ columns and ${2p+1 \choose p}$ rows from $A$. Hence, we have 
\begin{align*}
\rk(\psi_L(T)) & \geq \rk(A) - 2 {2p+1 \choose p}\\
& =  (m+1){2p+1 \choose p} - 2 {2p+1 \choose p} \\
& = {2p+1 \choose p}(2p).
\end{align*}
\end{proof}

Thus, we have 
\begin{align*}
\brk(T) &\geq \frac{{2p+1 \choose p}(2p) }{\rk(\mathcal{X}_L)}\\
& = \frac{{2p+1 \choose p}(2p) }{{2p \choose p}} \\
& > 2m-4.
\end{align*}

Hence $\brk(T) \geq 2m-3$ as required.

\end{document}